\documentclass{amsart}
\usepackage{amsthm, amssymb, array, enumerate, times, txfonts, %
mathrsfs, bbm, verbatim}
\usepackage[initials]{amsrefs}
\pdfoutput=1

\def\one{\mathbf1}
\def\zero{\mathbf0}
\def\rk{\text{rank}}
\providecommand{\meet}{\mathbin{\wedge}}
\providecommand{\join}{\mathbin{\vee}}
\newcommand{\card}[1]{\left| #1\right|}
\newcommand{\comp}[1]{\overline{#1}}
\ifx\upharpoonright\undefined
     \def\restrict{\hbox{\rm\kern0.166em\accent"12\kern-0.536em$\vert$\kern0.3em}}%
  \else
     \def\restrict{\upharpoonright}%
  \fi
\makeatletter
\def\twoSet#1#2{\left\{%
\vphantom{#2}#1\thinspace\right|\nolinebreak[3]\left.%
  #2%
  \vphantom{#1}%
  \right\}%
}
\def\oneSet#1{\left\lbrace#1\right\rbrace}

\newif\if@nstr
\def\setstrfalse{\let\if@nstr=\iffalse}
\def\setstrtrue{\let\if@nstr=\iftrue}
\def\@nstr #1#2{
\def\@@nstr ##1#1##2##3\@@nstr{\ifx
\@nstr ##2\setstrfalse \else \setstrtrue \fi }
\@@nstr #2#1\@nstr \@@nstr}
\def\@separate#1|#2@{\setFront{#1}\setBack{#2}}
\def\lb#1\rb{\@nstr|{#1} \if@nstr \@separate#1 @ \twoSet{\@setFront}{\@setBack}%
\else \@separate |{#1 }@ \oneSet{\@setBack}\fi%
}
\def\setFront#1{\def\@setFront{#1}}
\def\setBack#1{\def\@setBack{#1}}
\def\Set#1{\lb{#1}\rb}

\def\oneBrk#1{\left\langle#1\right\rangle}
\def\twoBrk#1#2{\left\langle%
\vphantom{#2}#1\thinspace\right|\nolinebreak[3]\left.%
  #2%
  \vphantom{#1}%
  \right\rangle%
}
\def\brk<#1>{\@nstr|{#1} \if@nstr \@separate#1 @ \twoBrk{\@setFront}{\@setBack}%
\else \@separate |{#1 }@ \oneBrk{\@setBack}\fi%
}

\def\lemref#1{\normalfont{lemma}~\ref{#1}}

\theoremstyle{plain}
\newtheorem{thm}{Theorem}[section]
\newtheorem{lem}[thm]{Lemma}
\newtheorem{cor}[thm]{Corollary}

\newtheorem{defn}[thm]{Definition}

\theoremstyle{remark}
{}
{}
{}
{}

\@ifpackageloaded{bbm}{

}{

}
\makeatother

\title{The M\"obius Function on Implication sublattices of a Boolean 
algebra}
\author{Colin G.Bailey}
\address{School of Mathematics,  Statistics \& Operations Research\\
Victoria University of Wellington\\
PO Box 600\\
Wellington\\
NEW ZEALAND
}
\email{Colin.Bailey@vuw.ac.nz}
\author{Joseph S.Oliveira}
\address{
Pacific Northwest National Laboratories\\
Richland\\
U.S.A.}
\email{Joseph.Oliveira@pnl.gov}
\date{2009,  February 3}

\begin{document}
\begin{abstract}
	Let $B$ be a finite Boolean algebra. Let $\mathcal A$ be the partial order 
	of all implication sublattices of $B$. We will compute the M\"obius 
	function on $\mathcal A$ in two different ways. 
\end{abstract}
\maketitle
\section{Introduction}

Let $B$ be a finite Boolean algebra.  
\begin{defn}\label{def:implsub}
	An \emph{implication subalgebra} of $B$ is a subset closed under $\to$. 
	
	An \emph{implication sublattice} of $B$ is a  subset closed under 
	$\to$  and $\wedge$. 
\end{defn}

Implication algebras are developed in \cite{Abb:bk}.
In this paper we wish to examine the poset of implication sublattices 
of a finite Boolean algebra. 
By the general theory of implication algebras we know that an 
implication sublattice is exactly a Boolean subalgebra of $[a, 1]$ for 
some $a\in B$,  and so we are really considering certain partial 
partitions of the atoms of $B$. 

%
%
%
%

Let $\mathcal A$ be the partial order of all implication sublattices 
of $B$ ordered by inclusion. Of course $\mathcal A$ is a finite lattice. 

Our interest is in understanding the M\"obius function of the poset 
$\mathcal A$. We consider two methods of finding it. Both methods use 
a closure operators that provide two ways of closing an implication 
sublattice to a Boolean subalgebra. 

\section{Method One}

We wish to compute the M\"obius function on $\mathcal A$. So let 
$A_{1}$ and $A_{2}$ be two implication sublattices of $B$. Since 
$A_{2}$ is a Boolean algebra and $A_{1}$ is an implication sublattice 
of $A_{2}$ we may assume that $A_{2}=B$. Let $A=A_{1}$ and $a=\min A$. 

First we define a closure operator on $\mathcal A$. 

\begin{defn}\label{def:closure}
	Let $A\in\mathcal A$. Let 
	$A^{c}=\Set{\comp x| x\in A}$. 
	
	Let $\comp A=A\cup A^{c}$. 
\end{defn}

\begin{lem}\label{lem:Bool}
	If $A\in\mathcal A$ then $\comp A$ is a Boolean subalgebra of $B$. 
\end{lem}
\begin{proof}
	Clearly $\comp A$ is closed under complements. As $1\in A$ we have 
	$1\in\comp A$ and $0\in\comp A$. It suffices to show closure under 
	joins. 
	
	If $x, y\in A$ then $x\join y\in A$ and $x\meet y\in A$ as $A$ is an 
	implication lattice. Thus if $x, y\in A^{c}$ then 
	$x\join y=\comp{\comp x\meet\comp y}\in A^{c}$. 
	
	If $x\in A$ and $y\in A^{c}$ then $x\join y= x\join \comp{\comp y}=
	\comp y\to x\in A$ as $\comp y\in A$ and $A$ is $\to$-closed. 
\end{proof}

\begin{lem}\label{lem:clAlg}
	Let $A\in\mathcal A$. Then $\comp A=A$ iff $A$ is a Boolean 
	subalgebra of $B$. 
\end{lem}
\begin{proof}
	By the last lemma $\comp A$ is a Boolean subalgebra. 
	
	If $A$ is a Boolean subalgebra then $A^{c}\subseteq A$ so that
	$\comp A=A$. 
\end{proof}

\begin{lem}\label{lem:closure}
	$A\mapsto\comp A$ is a closure operator on $\mathcal A$. 
\end{lem}
\begin{proof}
	Clearly $A\subseteq\comp A$. As $\comp A$ is a Boolean subalgebra we 
	have (by the last lemma) $\comp{\comp A}=\comp A$. 
\end{proof}

Now we recall the closure theorem for M\"obius functions -- see 
\cite{INC} Proposition 2.1.19. 
\begin{thm}\label{thm:closure}
	Let $X$ be a locally finite partial order and $x\mapsto \comp x$ be 
	a closure operator on $X$. Let $\comp X$ be the suborder of all 
	closed elements of $X$ and $y$, $z$ be in $X$. Then 
	$$
		\sum_{\comp x=\comp z}\mu(y, x)=
		\begin{cases}
			\mu_{\comp X}(y, \comp z) & \text{ if }y\in\comp X  \\
			0 & \text{ otherwise. }
		\end{cases}
	$$
\end{thm}
\begin{proof}
	 See \cite{INC}. 
\end{proof}

\begin{lem}\label{lem:closEqB}
	Let $C\in\mathcal A$. Then $\comp C=B$ iff $C=B$ or $C$ is an 
	ultrafilter of $B$. 
\end{lem}
\begin{proof}
	Suppose that $C$ is an ultrafilter. Then for any $x\in B$ we have 
	$x\in C$ or $x\in C^{c}$ so that $x\in\comp C$. Thus $\comp C=B$. 
	
	Suppose that $\comp C=B$ and $C\not=B$. We first show that 
	$C$ is upwards-closed. Indeed, if not, then there is some 
	$x\in B$ and $b\le x<1$ with $b\in C$ and $x\notin C$. As 
	$\comp C=B$ we have $\comp x\in C$ so that $\comp x\to b\in C$. 
	But $\comp x\to b= \comp{\comp x}\join b= x\join b= x$ -- 
	contradiction. 
	
	Thus $C$ is upwards-closed, and meet and join-closed, so $C$ is a 
	filter. As $C\not=B$ we know that $0\notin C$. Also 
	$C\cup C^{c}=B$ so that for all $x\in C$ either $x\in C$ or $\comp 
	x\in C$. Thus $C$ is an ultrafilter. 
\end{proof}

This lemma together with the closure theorem allow us to use an 
induction argument to compute the M\"obius function. The induction 
comes from the following lemma. 

\begin{lem}\label{lem:ufsAbove}
	Let $A\in\mathcal A$ and $a=\min A$. Let $c_{1}$ and $c_{2}$ be any 
	atoms below $a$. Then 
	$$
		\bigl[A, [c_{1}, 1]\bigr]_{\mathcal A}\simeq\bigl[A, [c_{2}, 1]\bigr]_{\mathcal A}. 
	$$
\end{lem}
\begin{proof}
	Let $\tau$ be the permutation of the atoms of $B$ that exchanges 
	$c_{1}$ and $c_{2}$. Then $\tau$ induces an automorphism of $B$ and 
	that induces an automorphism of $\mathcal A$. It is clear that this 
	induces the desired isomorphism between 
	$\bigl[A, [c_{1}, 1]\bigr]_{\mathcal A}$ and $\bigl[A, [c_{2}, 1]\bigr]_{\mathcal A}$. 
\end{proof}

Now suppose that $A\in\mathcal A$ and $a=\min A>0$. Then we have 
$$
	\mu(A, B)+\sum_{
	\substack{c\le a\\
	c\text{ a }B\text{-atom}}}\mu(A, [c, 1])=0
$$
by the closure theorem and \lemref{lem:closEqB}. 
Thus 
\begin{align*}
	\mu(A, B) & =-\sum_{\substack{c\le a\\
	c\text{ a }B\text{-atom}}}\mu(A, [c, 1])  \\
	 & =-\card a \mu(A, [c, 1])\\
	 \intertext{where $c$ is any $B$-atom below $a$. $\card a$ is the rank of $a$ in $B$ 
	 and equals the number of atoms below $a$. As we now have a reduction 
	 in rank (of $a$ in $[c, 1]$) we see that we can proceed inductively to 
	 get }
	 =-1^{\card a}\card a!\mu(A, [a, 1]). 
\end{align*}

So we are left with the case that $A$ is in fact a Boolean subalgebra 
of $B$. We note that in this case, if $C\in[A, B]_{\mathcal A}$ then 
$C$ is also a Boolean subalgebra. We also note that any subalgebra is 
determined by its set of atoms and these form a partition of $n$. 
So the lattice of subalgebras of the Boolean algebra $2^{n}$ is isomorphic
to the lattice of partitions of $n$ and the M\"obius function of this 
is well known. This gives us the final result that 
$$
\mu(A, B)=(-1)^{\card a+\text{w}(A)-\text{w}(B)}\card 
a!\prod_{\substack{c\text{ is an }\\
A\text{-atom}}}({\card c-\card a-1})! 
$$
where $\text{w}(A)$ is the number of atoms of $A$ and $\card c$ is the 
$B$-rank of $c$. 

\section{Method Two}
Consider any implication sublattice $A$ of $B$. Then $A$ is a Boolean 
subalgebra of $[a=\min A, \one]$. This means we can take any extension 
$A\subseteq C\subseteq B$ and factor $C$ into $\brk<[a, \one]\cap C, [0,  
a]\cap C>$ and this pairing completely determines $C$.  

It follows that the interval $[A, B]$ is isomorphic to a product of 
two partial orders:
\begin{align*}
	P_{1} & =\Set{C | C\text{ is a Boolean subalgebra of }[a, 1]\text{ containing 
	}[a, 1]\cap A }  \\
	P_{2} & =\Set{C | C\text{ is an implication sublattice of }[0, a]}. 
\end{align*}
$P_{1}$ is known as a partition lattice. $P_{2}$ is essentially the 
same as the interval we are considering with the assumption that 
$A=\Set{\one}$. 

So we will compute $\mu(\Set{\one}, B)$. 

\begin{defn}\label{def:upClose}
	Let $C$ be any implication sublattice of $B$. Then 
	$$
		C\uparrow=\Set{x | \exists c\in C\ x\geq c}
	$$
	is the upwards-closure of $C$. 
	Note that $C\uparrow=[\min c, \one]$. 
\end{defn}

\begin{lem}\label{lem:upClose}
	$C\mapsto C\uparrow$ is a closure operator and $C\uparrow=B$ iff 
	$\min C=0$. 
\end{lem}
\begin{proof}
	This is immediate. 
\end{proof}

It follows from this lemma that $C\uparrow=B$ iff $C$ is a Boolean 
subalgebra of $B$. 

\begin{lem}\label{lem:isoUp}
	Let $C_{1}$ and $C_{2}$ be two Boolean subalgebras of $B$. Then 
	$$
		[\Set{\one}, C_{1}]\simeq [\Set{\one}, C_{2}]\text{ iff 
		}C_{1}\simeq C_{2}. 
	$$
\end{lem}
\begin{proof}
	The left-to-right direction is clear. 
	
	Conversely, if $1>s_{i1}>s_{i2}>\dots>s_{ij_{i}}=0$ is a maximal 
	chain in $C_{i}$ then 
	the set $\Set{[s_{ij}, \one] | 1\le j\le j_{i}}$ is a maximal chain 
	in $[\Set{\one}, C_{i}]$ -- since 
	$[s_{i(j+1)}, \one]$ has one more atom than $[s_{ij}, \one]$. 
	
	Thus $j_{1}=j_{2}$ and so $C_{1}\simeq C_{2}$. 
\end{proof}

We recall that there are $S_{n, k}$ Boolean subalgebras of $B$ that 
have $k$ atoms -- here $n$ is the number of atoms that $B$ has and 
$S_{n, k}$ is a Stirling number of the second kind, counting the 
number of partitions of $n$ into $k$ pieces. 

Let $C_{k}$ be any Boolean subalgebra of $B$ with $k$ atoms. 
We can now apply the lemma and the closure theorem to see that 
\begin{align*}
	\mu(\Set{\one}, B) 
	&=-\sum_{\substack{C\uparrow=B\\
	C\not=B}}\mu(\Set{\one}, C) \\
	&=-\sum_{n>k\geq 1}S_{n, k}\mu(\Set{\one}, C_{k})  \\
	 & =\sum_{\substack{\Gamma\text{ a chain in }[1, n]\\
	 \Gamma=n=n_{0}>n_{1}>\dots>n_{p}=1}}(-1)^{p}S_{n_{0}, 
	 n_{1}}\cdot\dots\cdot S_{n_{p-1}, n_{p}}\mu(\Set{\one}, C_{1}) \\
	 &  =\sum_{\substack{\Gamma\text{ a chain in }[1, n]\\
	 \Gamma=n=n_{0}>n_{1}>\dots>n_{p}=1}}(-1)^{p+1}S_{n_{0}, 
	 n_{1}}\cdot\dots\cdot S_{n_{p-1}, n_{p}} 
\end{align*}

\subsection{An Identity}
We can put these two methods together to see that 
$$
\mu(\Set{\one}, B)=(-1)^{n}n!=\sum_{\substack{\Gamma\text{ a chain in }[1, n]\\
	 \Gamma=n=n_{0}>n_{1}>\dots>n_{p}=1}}(-1)^{p+1}S_{n_{0}, 
	 n_{1}}\cdot\dots\cdot S_{n_{p-1}, n_{p}}. 
$$

\subsection{Conclusion \& Beyond}
We see fro the above results that the poset $\mathcal A$ is close to 
the poset of partitions of a set. The analysis we've undertaken shows 
this in two distinct ways -- via the closure operators. In future 
work we plan to apply these results to an analysis of the subalgebras 
of cubic implication algebras.

\end{document}